\documentclass[reqno]{amsart}
\usepackage[margin=1in]{geometry}
\usepackage[english]{babel}
\usepackage{changepage}
\usepackage{enumitem}
\usepackage{amsmath, amsthm, amssymb, amsfonts,tikz,mathtools,comment}
\usepackage{multicol}
\usepackage{float}
\newtheorem{theorem}{Theorem}[section]
\newtheorem{corollary}[theorem]{Corollary}
\newtheorem{lemma}[theorem]{Lemma}
\newtheorem{conjecture}[theorem]{Conjecture}

\theoremstyle{definition}

\usepackage{xcolor}

\title{Nordhaus--Gaddum Inequalities for Dominating-Set Counts in Bipartite Graphs}
\author{T. N. Sanh}
\address{Department of Computer Science, Princeton University}
\email{tai.sanh.n@princeton.edu}
\date{} 

\begin{document}

\maketitle

\begin{abstract}
    A dominating set in a graph $G$ is a subset $S$ of its vertices such that each vertex in $G$ is either in $S$ or adjacent to a vertex in $S$. Nordhaus--Gaddum inequalities relate the values of a graph parameter on a graph and its complement. In this setting, Keough and Shane conjecture that any graph $G$ on $n$ vertices satisfies $\partial(G) + \partial(\bar{G}) \leq 2(2^{\lfloor n/2 \rfloor} - 1)(2^{\lceil n/2 \rceil} - 1) + 2$, where $\partial(G)$ is the number of dominating sets in $G$. We partially resolve this conjecture for the bipartite case by proving the stronger bound: for a bipartite graph $G$ with nonempty bipartition $(A,B)$, it holds that $\partial(G) + \partial(\bar{G}) \leq 2(2^{|A|} - 1)(2^{|B|} - 1) + 2$. We also characterize the bipartite graphs for which equality holds.
\end{abstract}

\section{Introduction}

For a graph $G$, we call a subset of its vertices $S$ {\sl dominating} if all vertices in $G$ are either in $S$ or adjacent to a vertex in $S$. Denote by $\partial(G)$ the number of dominating sets in $G$. Nordhaus--Gaddum inequalities relate the values of a graph parameter
on a graph and its complement. The original Nordhaus--Gaddum inequalities established bounds on the sum and product of the chromatic numbers of a graph and its complement~\cite{NordhausGaddum}. Since then, analogous inequalities have been studied for a wide variety of graph parameters. Recent work has also considered enumerative versions of such problems, in which the parameter counts a family of graph substructures. For example, Bal, Cutler, and Pebody~\cite{BalCutlerPebody} studied Nordhaus--Gaddum inequalities for the number of cliques in a graph.

The number of dominating sets provides another natural counting parameter in this setting. For this parameter, Wagner~\cite{Wagner} showed that every graph $G$ on $n$ vertices satisfies
$$
\partial(G) + \partial(\bar{G}) \geq 2^n.
$$
Keough and Shane subsequently established the following upper bound.
\begin{theorem}[\cite{KeoughShane}]
    For a graph $G$ on $n$ vertices, it holds that
    \begin{equation*}
        \partial(G) + \partial(\bar{G}) \leq 2^{n+1} - 2^{\lfloor n/2 \rfloor} - 2^{\lceil n/2 \rceil - 1}.
    \end{equation*}
\end{theorem}
\noindent They note that this bound is not expected to be tight and conjecture the following stronger upper bound on any graph $G$.

\begin{conjecture}[\cite{KeoughShane}]\label{con:main}
    For a graph $G$ on $n$ vertices, it holds that
    \begin{equation*}
        \partial(G) + \partial(\bar{G}) \leq 2(2^{\lfloor n/2 \rfloor} - 1)(2^{\lceil n/2 \rceil} - 1) + 2.
    \end{equation*}
\end{conjecture}

\noindent Instead of fully resolving the conjecture, we prove it for the bipartite case. In fact, we show a stronger bound with the following theorem.
\begin{theorem}\label{thm:main}
    For a bipartite graph $G$ with nonempty bipartition $(A,B)$, it holds that
    \begin{equation*}
        \partial(G) + \partial(\bar{G}) \leq 2(2^{|A|} - 1)(2^{|B|} - 1) + 2.
    \end{equation*}
\end{theorem}

\noindent Observe that the right-hand side is maximized when both parts are as balanced as possible. That is, if $n = |A| + |B|$, then we have
\begin{align*}
    \partial(G) + \partial(\bar{G}) & \leq 2(2^{|A|} - 1)(2^{|B|} - 1) + 2                                \\
                                & \leq 2(2^{\lfloor n/2 \rfloor} - 1)(2^{\lceil n/2 \rceil} - 1) + 2,
\end{align*}
\noindent which is exactly the bound in Conjecture~\ref{con:main}. We
formalize this claim in Corollary~\ref{cor:conj}.

\section{The Bipartite Case}

For any graph $G$ on $n$ vertices, Wagner~\cite{Wagner} showed that
$\partial(G) + \partial(\bar{G}) \geq 2^n$. This lower bound follows from the fact that
for any $S \subseteq V(G)$, we have that $S$ dominates $G$ or $\bar{S}$ dominates $\bar{G}$,
where $\bar{S}=V(G)\setminus S$. 
With this, Keough and Shane~\cite{KeoughShane} express the number of dominating sets
in $G$ and $\bar{G}$ as
\begin{equation}\label{ineq:one}
    \partial(G) + \partial(\bar{G}) = 2^n + |\Upsilon(G, \bar{G})|,
\end{equation}
where they define
\begin{equation*}
    \Upsilon(G, \bar{G}) = \{ S \subseteq V(G) : S \text{ dominates } G \text{ and } \bar{S} \text{ dominates } \bar{G} \}.
\end{equation*}

\noindent Let us now consider only bipartite $G$ with bipartition $(A,B)$ and $|A|, |B| \geq 2$.
We will later address the case where $|A| = 1$ or $|B| = 1$. To prove Theorem~\ref{thm:main}, we then must show
\begin{equation*}
    |\Upsilon(G, \bar{G})| \leq (2^{|A|} - 2)(2^{|B|} - 2).
\end{equation*}

\noindent Let us define $\mathcal{P} = \left\{ (X,Y) : \emptyset \neq X \subsetneq A,\ \emptyset \neq Y \subsetneq B \right\}$.
Then $|\mathcal{P}| = (2^{|A|}-2)(2^{|B|}-2)$.
Let \(\mathcal{I}\subseteq\mathcal{P}\) consist of those pairs
\((X,Y)\) such that \(X\cup Y\) dominates \(G\).
We first identify the members of \(\Upsilon(G,\bar{G})\) that
contain neither \(A\) nor \(B\) entirely. For each
\(S\subseteq V(G)\), we write
\[
    X=S\cap A
    \ \ \text{and}\ \
    Y=S\cap B.
\]
Suppose that \(S\in\Upsilon(G,\bar{G})\), with
\(X\subsetneq A\) and \(Y\subsetneq B\). Since \(S\) dominates \(G\),
both \(X\) and \(Y\) must be nonempty. Indeed, if \(X=\emptyset\),
then every vertex of \(B\setminus Y\) is undominated, since \(G\)
has no edges within \(B\). The case \(Y=\emptyset\) is symmetric.
Thus \((X,Y)\in\mathcal{P}\), and thus \((X,Y)\in\mathcal{I}\).
Conversely, suppose that \((X,Y)\in\mathcal{I}\). Then \(X\cup Y\)
dominates \(G\). Moreover, it follows that
$$
    \overline{X\cup Y} = (A\setminus X)\cup(B\setminus Y)
$$
intersects both \(A\) and \(B\) since $X \subsetneq A$ and $Y \subsetneq B$. Since both \(A\) and \(B\) induce
cliques in \(\bar{G}\), this complement dominates
\(\bar{G}\). Therefore, we have \(X\cup Y\in\Upsilon(G,\bar{G})\).
It remains to account for the members of \(\Upsilon(G,\bar{G})\)
that contain one part of the bipartition entirely. Let us define
$$
    \mathcal{L}
    =
    \left\{
    L\subsetneq A :
    L\cup B\in\Upsilon(G,\bar{G})
    \right\}
    \ \ \text{and} \ \
    \mathcal{R}
    =
    \left\{
    R\subsetneq B :
    A\cup R\in\Upsilon(G,\bar{G})
    \right\}.
$$

\noindent Every member of $\Upsilon(G,\bar{G})$ corresponds uniquely either
to a pair in $\mathcal I$, to an element $L\in\mathcal L$ via
$L\mapsto L\cup B$, or to an element $R\in\mathcal R$ via
$R\mapsto A\cup R$. The only remaining possibility, $X=A$ and $Y=B$,
gives $S=A\cup B=V(G)$, which is not in $\Upsilon(G,\bar{G})$ since
$\overline{A \cup B} = \emptyset$ cannot dominate nonempty $\bar{G}$.
Thus we have
\begin{equation}\label{ineq:decomposition}
    |\Upsilon(G,\bar{G})|
    =
    |\mathcal{I}|+|\mathcal{L}|+|\mathcal{R}|.
\end{equation}

\medskip
\begin{lemma}\label{lem:main}
    Assume $|A|, |B| \geq 2$. It holds that
    $|\mathcal{P}| - |\mathcal{I}| \geq |\mathcal{L}| + |\mathcal{R}|$.
\end{lemma}
\begin{proof}
    Note that $\mathcal{P} \setminus \mathcal{I}$ is the set of pairs in $\mathcal{P}$ that do not
    dominate $G$. We first show that $|\mathcal{P}| - |\mathcal{I}| \geq |A| |\mathcal{R}|$.
    Fix $a \in A$. Let us construct an injection from $\mathcal{R}$ to
    $\mathcal{P} \setminus \mathcal{I}$ whose $A$-coordinate is $\{a\}$.
    For every nonempty $R \in \mathcal{R}$, let us define
    \begin{equation*}
        \phi_a(R) = (\{a\}, R).
    \end{equation*}
    
    \noindent Since $\emptyset \neq R \subsetneq B$ and $\{a\} \subsetneq A$, it follows that
    $\phi_a(R) \in \mathcal{P}$. Let us now show that $\phi_a(R) \notin \mathcal{I}$.
    By definition of $\mathcal{R}$, we have
    $S = A \cup R \in \Upsilon(G, \bar{G})$. By definition of
    $\Upsilon(G, \bar{G})$, it follows that
    $\bar{S} = B \setminus R$ dominates $\bar{G}$. Since
    $a \notin B \setminus R$, it follows that there exists some
    $b \in B \setminus R$ such that $ab \in E(\bar{G})$.
    But that means $ab \notin E(G)$. So the set $\{a\} \cup R$ cannot dominate $G$, since
    $b \notin R$, $b$ is not adjacent to any vertex in $R$ due to bipartiteness, and
    $a$ is not adjacent to $b$ in $G$. Thus $\phi_a(R) \notin \mathcal{I}$.

    \noindent It remains to define $\phi_a(\emptyset)$ if $\emptyset \in \mathcal{R}$.
    In this case, $A \in \Upsilon(G,\bar{G})$, so $B$ dominates $\bar{G}$.
    Since $a \notin B$, there exists some $b_a \in B$ such that
    $ab_a \in E(\bar{G})$. Define
    \begin{equation*}
        \phi_a(\emptyset) = (\{a\}, B \setminus \{b_a\}).
    \end{equation*}
    
    \noindent Since $|A|,|B| \geq 2$, it follows that
    $\phi_a(\emptyset) \in \mathcal{P}$. Moreover, $b_a$ is not dominated by
    $\{a\} \cup (B \setminus \{b_a\})$ in $G$, so
    $\phi_a(\emptyset) \notin \mathcal{I}$.

    \noindent We must also show that this image does not coincide with the image of a nonempty
    member of $\mathcal{R}$. Suppose otherwise that
    $B \setminus \{b_a\} \in \mathcal{R}$. Then
    $A \cup (B \setminus \{b_a\}) \in \Upsilon(G,\bar{G})$, so
    $\{b_a\}$ dominates $\bar{G}$. Thus $b_a$ is adjacent in $\bar{G}$ to every
    vertex of $A$, and thus has no neighbor in $A$ in $G$. However,
    $A \cup (B \setminus \{b_a\})$ dominates $G$, so $b_a$
    must have a neighbor in $A$. Contradiction! Therefore, $\phi_a$ is injective.

    \noindent The images for different $a$'s are disjoint because their first coordinates differ.
    Therefore, we have
    $|\mathcal{P}| - |\mathcal{I}| \geq |A| |\mathcal{R}|$.
    By symmetry, it holds that
    $|\mathcal{P}| - |\mathcal{I}| \geq |B| |\mathcal{L}|$. Then we have that
    \[
        |\mathcal{P}| - |\mathcal{I}|
        \geq
        \max\{ |B||\mathcal{L}|, |A||\mathcal{R}| \}.
    \]
    It remains to show that
    \[
        \max\{ |B||\mathcal{L}|, |A||\mathcal{R}| \}
        \geq
        |\mathcal{L}| + |\mathcal{R}|.
    \]
    If $|\mathcal{L}|=0$ or $|\mathcal{R}|=0$, the inequality follows immediately.
    Thus, assume that $|\mathcal{L}|,|\mathcal{R}|>0$. Suppose otherwise that
    $|B||\mathcal{L}| < |\mathcal{L}| + |\mathcal{R}|$ and
    $|A||\mathcal{R}| < |\mathcal{L}| + |\mathcal{R}|$.
    Then we have that
    $$
    |\mathcal{L}|(|B| - 1) < |\mathcal{R}| \ \ \text{and} \ \    
    |\mathcal{R}|(|A| - 1) < |\mathcal{L}|.$$
    
    \noindent Multiplying gives
    \[
        |\mathcal{L}||\mathcal{R}|(|A| - 1)(|B| - 1)
        <
        |\mathcal{L}||\mathcal{R}|,
    \]
    and thus
    \[
        (|A| - 1)(|B| - 1) < 1.
    \]
    
    \noindent But we assumed that $|A|, |B| \geq 2$. Contradiction!
\end{proof}

\noindent We now present the proof of Theorem~\ref{thm:main}.

\begin{proof}[Proof of Theorem~\ref{thm:main}]
    First assume $|A|, |B| \geq 2$.
    Using Lemma~\ref{lem:main} and (\ref{ineq:decomposition}), we get that
    $$
        |\Upsilon(G, \bar{G})| \leq |\mathcal{P}| = (2^{|A|} - 2)(2^{|B|} - 2),
    $$
    as desired. For $|A| = 1$, we claim that $\Upsilon(G, \bar{G}) = \emptyset$. Indeed, suppose
    otherwise that there exists $S \in \Upsilon(G, \bar{G})$. Let $A = \{a\}$. Then we have two cases.
    \begin{adjustwidth}{2em}{0pt}
    \textbf{Case 1.} \(a \in S\).
    By definition, \(\bar{S}\) dominates \(\bar{G}\). But \(a \notin \bar{S}\), so there must exist some \(b \in \bar{S}\)
    such that \(ab \in E(\bar{G})\). But that means \(ab \notin E(G)\). So \(S\) cannot dominate \(G\), since
    \(b \notin S\), \(b\) is not adjacent to any vertices in \(B\) due to bipartiteness, and \(a\) is not adjacent to \(b\)
    in \(G\). But we assumed \(S \in \Upsilon(G,\bar{G})\), which means \(S\) must dominate \(G\). Contradiction!
    
    \medskip
    
    \noindent \textbf{Case 2.} \(a \notin S\).
    Since $S$ dominates $G$, we must have $S = B$. Indeed, if there exists some $b \in B \setminus S$, then $b$ is not adjacent to any vertices in $S \subseteq B$ due to bipartiteness, so $b$ is not dominated by $S$. Since $a \notin S$, there must exist some $b \in S$ such that $ab \in E(G)$. But $\bar{S} = \{a\}$ dominates $\bar{G}$. Since $b \notin \bar{S}$, it follows that $ab \in E(\bar{G})$. But that means $ab \notin E(G)$. Contradiction!
    \end{adjustwidth}

    \noindent Therefore, for $|A| = 1$, we get $|\Upsilon(G, \bar{G})| = 0 \leq |\mathcal{P}| = (2^{|A|} - 2)(2^{|B|} - 2) = 0$.
    The case for $|B| = 1$ is symmetric. Finally, we plug into (\ref{ineq:one}), giving us
    \begin{align*}
        \partial(G) + \partial(\bar{G}) & = 2^n + |\Upsilon(G, \bar{G})|        \\
                                    & \leq 2^n + (2^{|A|} - 2)(2^{|B|} - 2) \\
                                    & = 2(2^{|A|} - 1)(2^{|B|} - 1) + 2.
    \end{align*}
\end{proof}

\noindent Now let us show that Theorem~\ref{thm:main} partially resolves Conjecture~\ref{con:main} for the bipartite case.

\begin{corollary}\label{cor:conj}
    For a bipartite graph $G$ with nonempty bipartition $(A,B)$
    and $n = |A| + |B|$, it holds that
    $$
        \partial(G) + \partial(\bar{G}) \leq 2(2^{\lfloor n/2 \rfloor} - 1)(2^{\lceil n/2 \rceil} - 1) + 2.
    $$
\end{corollary}

\begin{proof}
    By Theorem~\ref{thm:main}, it holds that
    $$
        \partial(G) + \partial(\bar{G})
        \leq
        2(2^{|A|} - 1)(2^{|B|} - 1) + 2.
    $$
    It remains to show that
    $$
        (2^{|A|} - 1)(2^{|B|} - 1)
        \leq
        (2^{\lfloor n/2 \rfloor} - 1)(2^{\lceil n/2 \rceil} - 1).
    $$
    
    \noindent Assume without loss of generality that $|A| \leq |B|$. If
    $|B|-|A|\geq 2$, then
    $$
        (2^{|A|+1}-1)(2^{|B|-1}-1)
        -
        (2^{|A|}-1)(2^{|B|}-1)
        =
        2^{|B|-1}-2^{|A|}
        >0.
    $$
    Thus, moving one vertex from the larger part to the smaller part
    increases the product. Repeating this process gives that the product
    is maximized when the two part sizes differ by at most one, that is, at
    $|A|=\lfloor n/2\rfloor$ and $|B|=\lceil n/2\rceil$. The corollary follows.
\end{proof}

\noindent The following corollary follows immediately from a result of Brouwer,
Csorba, and Schrijver~\cite{BrouwerCsorbaSchrijver}, which states that every nonempty graph has an odd number of dominating sets.
\begin{corollary}\label{cor:main}
    For a bipartite graph $G$ with nonempty bipartition $(A,B)$ and $n = |A| + |B|$, it holds that
    \begin{align*}
        \partial(G) \partial(\bar{G}) &\leq (2^{|A|} - 1)(2^{|B|} - 1)((2^{|A|} - 1)(2^{|B|} - 1) + 2) \\
        &\leq (2^{\lfloor n/2 \rfloor} - 1)(2^{\lceil n/2 \rceil} - 1)((2^{\lfloor n/2 \rfloor} - 1)(2^{\lceil n/2 \rceil} - 1) + 2)
    \end{align*}
\end{corollary}
\begin{proof}
    Let $q = (2^{|A|} - 1)(2^{|B|} - 1)$. Theorem~\ref{thm:main} gives us that $\partial(G) + \partial(\bar{G}) \leq 2q + 2$.
    Since $\partial(G)$ and $\partial(\bar{G})$ are odd
    and since their sum is at most $2q+2$, their product is maximized when their values are $q$ and $q+2$.
    The second inequality follows by the same argument as in the proof of
    Corollary~\ref{cor:conj}.
\end{proof}

\noindent The first inequality is tight since for $G = K_{|A|,|B|}$ for some nonempty bipartition $(A,B)$,
we get $\partial(K_{|A|,|B|}) = (2^{|A|} - 1)(2^{|B|} - 1) + 2$ and $\partial(\overline{K_{|A|,|B|}}) = (2^{|A|} - 1)(2^{|B|} - 1)$.
The second inequality is tight with $G = K_{|A|,|B|}$ for some nonempty bipartition $(A,B)$ where $|A|$ and $|B|$ differ by at most one.

\section{Extremal Bipartite Graphs}

We saw that the bound in Theorem~\ref{thm:main} is tight for every choice of nonempty part sizes, since it is attained by the complete bipartite graph $K_{|A|,|B|}$. In fact, we can characterize all bipartite graphs that attain equality.

\begin{theorem}\label{thm:equality}
    For a bipartite graph $G$ with nonempty bipartition $(A, B)$, it holds that
    $$
        \partial(G) + \partial(\bar{G}) = 2(2^{|A|} - 1)(2^{|B|} - 1) + 2,
    $$
    if and only if one of the following is satisfied:
    \begin{enumerate}[
    label={\makebox[2em][c]{\normalfont(\roman*)}},
    labelwidth=2em,
]
    \item $\min\{ |A|, |B| \} = 1$.
    \item $|A| = |B| = 2$ and $G \cong 2K_2$.
    \item $G \cong K_{|A|,|B|}$.
    \end{enumerate}
\end{theorem}
\begin{proof}
    We use the same notation as in Lemma~\ref{lem:main}.
    Suppose first that equality holds. If
$\min\{|A|,|B|\}=1$, then condition (i) holds. So assume that
    $|A|,|B|\geq 2$. By \eqref{ineq:decomposition}, we have
$$
|\Upsilon(G,\bar{G})|
=
|\mathcal{P}|-\bigl(|\mathcal{P}|-|\mathcal{I}|\bigr)
+|\mathcal{L}|+|\mathcal{R}|.
$$
Therefore, equality in Theorem~\ref{thm:main} holds if and only if
$$
|\mathcal{P}|-|\mathcal{I}|
=
|\mathcal{L}|+|\mathcal{R}|.
$$

\noindent From the proof of Lemma~\ref{lem:main}, we have
$$
|\mathcal{P}|-|\mathcal{I}|
\geq |A||\mathcal{R}|
\ \ \text{and}\ \
|\mathcal{P}|-|\mathcal{I}|
\geq |B||\mathcal{L}|.
$$

\noindent Suppose first that $|\mathcal{L}|=0$ or $|\mathcal{R}|=0$. Assume wlog. that $|\mathcal{L}|=0$. Then we have
$|\mathcal{P}|-|\mathcal{I}|=|\mathcal{R}|$. Since
$|\mathcal{P}|-|\mathcal{I}| \geq |A||\mathcal{R}|$
and $|A|\geq2$, it follows that $|\mathcal{R}|=0$, and therefore
$|\mathcal{P}|-|\mathcal{I}|=0.$
Thus, we have
$$
\mathcal{I}=\mathcal{P}.
$$

\noindent We claim that $G\cong K_{|A|,|B|}$. Fix any $a\in A$ and $b\in B$.
Since $|A|\geq2$, choose some $a'\in A\setminus \{a\}$. The pair
$(\{a'\},\{b\})$ belongs to $\mathcal{P}$, and thus belongs to
$\mathcal{I}$. Thus $\{a',b\}$ dominates $G$. Since $a$ does not
belong to this set and there are no edges within $A$, it follows that
$a$ must be adjacent to $b$. Since $a$ and $b$ were arbitrary, every
vertex of $A$ is adjacent to every vertex of $B$. Therefore, it follows that
$$
G\cong K_{|A|,|B|}.
$$

\noindent Now suppose that $|\mathcal{L}|,|\mathcal{R}|>0$. Since
$|\mathcal{P}|-|\mathcal{I}| = |\mathcal{L}|+|\mathcal{R}|$,
it follows that the inequalities
$$
|\mathcal{P}|-|\mathcal{I}|
\geq |A||\mathcal{R}|
\ \ \text{and}\ \ 
|\mathcal{P}|-|\mathcal{I}|
\geq |B||\mathcal{L}|
$$
give us
$$
|\mathcal{L}|
\geq (|A|-1)|\mathcal{R}|
\ \ \text{and}\ \ 
|\mathcal{R}|
\geq (|B|-1)|\mathcal{L}|.
$$

\noindent Now, multiplying these inequalities gives
$$
|\mathcal{L}||\mathcal{R}|
\geq
(|A|-1)(|B|-1)|\mathcal{L}||\mathcal{R}|.
$$

\noindent Since $|\mathcal{L}|,|\mathcal{R}|>0$, it follows that
$(|A|-1)(|B|-1)\leq 1.$ Since $|A|,|B|\geq2$, we must have
$$
|A|=|B|=2.
$$

\noindent We now show that $G\cong 2K_2$. We write $A=\{a_1,a_2\}$ and $B = \{b_1,b_2\}$. We first claim that no singleton belongs to $\mathcal{R}$. Suppose,
for example, that $\{b_1\}\in\mathcal{R}$. Then we have
$$
A\cup\{b_1\}\in\Upsilon(G,\bar{G}),
$$
so its complement $\{b_2\}$ dominates $\bar{G}$. Thus $b_2$ is
adjacent in $\bar{G}$ to every vertex of $A$, and thus $b_2$ has no
neighbor in $A$ in $G$. However, $A\cup\{b_1\}$ dominates $G$, so $b_2$ must have a neighbor in $A$. Contradiction! Therefore, no singleton belongs to $\mathcal{R}$. Since
$\mathcal{R}\neq\emptyset$, it follows that
$\mathcal{R}=\{\emptyset\}$.
By symmetry, we also have $\mathcal{L}=\{\emptyset\}$.

\noindent Thus both $A$ and $B$ belong to $\Upsilon(G,\bar{G})$. Since $A$
dominates $G$, every vertex of $B$ has at least one neighbor in $A$.
Since $B$ dominates $\bar{G}$, every vertex of $A$ has at least one
nonneighbor in $B$, and thus has degree at most one in $G$.
Similarly, since $B$ dominates $G$, every vertex of $A$ has at least
one neighbor in $B$, while the fact that $A$ dominates $\bar{G}$
implies that every vertex of $B$ has degree at most one in $G$.
Therefore, every vertex of $G$ has degree exactly one, and thus we get
$$
G\cong 2K_2.
$$

\noindent Conversely, suppose that one of conditions (i), (ii), (iii) holds. If $\min\{|A|,|B|\}=1$, then as shown in the proof of Theorem~\ref{thm:main}, we get
$$
\Upsilon(G,\bar{G})=\emptyset.
$$
And thus, we get
$$
\partial(G)+\partial(\bar{G})
=
2^{|A|+|B|}
=
2(2^{|A|}-1)(2^{|B|}-1)+2.
$$

\noindent If $G\cong K_{|A|,|B|}$, then a dominating set of $G$ either
intersects both $A$ and $B$, or is equal to $A$ or $B$. Thus
$$
\partial(G)
=
(2^{|A|}-1)(2^{|B|}-1)+2.
$$
Moreover, it follows that $\bar{G}\cong K_{|A|}\cup K_{|B|}$,
so a dominating set of $\bar{G}$ must intersect both components. So, we have that
$$
\partial(\bar{G})
=
(2^{|A|}-1)(2^{|B|}-1),
$$
and equality follows. Finally, suppose that $G\cong2K_2$. Then
$ \partial(G)=9$ and $\partial(\bar{G})=11$. Since $|A|=|B|=2$, we have
$$
\partial(G)+\partial(\bar{G})
=
20
=
2(2^2-1)(2^2-1)+2.
$$
\end{proof}

\section{Conclusion}
In this paper, we partially resolve the conjecture of Keough and Shane by proving it for all bipartite graphs. In fact, we establish a stronger bound in terms of the sizes of the two parts of a bipartition and characterize the bipartite graphs for which equality holds. It remains open whether the conjectured bound holds for all graphs.

\section*{Acknowledgements}

I would like to thank Lauren Keough for reading an earlier version of this paper and for her kind feedback and suggestions.

\bibliographystyle{unsrt}
\bibliography{references}
\end{document}